\documentclass{amsproc}

\newtheorem{theorem}{Theorem}[section]
\newtheorem{lemma}[theorem]{Lemma}

\theoremstyle{definition}
\newtheorem{definition}[theorem]{Definition}

\theoremstyle{remark}
\newtheorem{remark}[theorem]{Remark}

\numberwithin{equation}{section}

\usepackage{amssymb, xcolor, enumerate}
\usepackage[
backend=biber,
style=alphabetic,
sorting=ynt,
doi=false,
isbn=false,
url=false,
eprint=false
]{biblatex}
\title{Noncommutative algebraic sets are not generically nc manifolds}
\author{Jiachen Lu}

\begin{document}
	
	\begin{abstract}
		A consequence of implicit function theorem is that any algebraic set is a complex manifold except at some singular points which form a strictly lower dimension algebraic subset. Accordingly, an non commutative implicit function theorem of Agler-McCarthy might suggest that a nc algebraic set is genetically an nc manifold. However, in this paper, a type of nc polynomials is shown, whose zero locus cannot be nc manifolds.
	\end{abstract}
    \maketitle
    \\
    \AmS{ 15A54, 58J42, 47J07.}

	\section{Introduction}
		The study of solutions to polynomial equations in matrix variables has been actively developed in recent years. For example, a free model theory was developed and enriched for graded nc functions in \cite{ncokaweil,ncrealization2,ncrealization3} to show the non-commutative version of Pick's interpolation and the Oka-Weil Theorem. And the nc Oka-weil theorem, that a free holomorphic function can be uniformly approximated by free polynomials, enabled lots of results to be built on the non-commutative setting, such as Pick's interpolation for free holomorphic functions \cite{ncPickInter} and the representation of free Herglotz functions \cite{ncHerglotz}
		\newline
		Apart from the complex analysis point of view, there are also interesting results in the nc setting of differential geometry. An nc manifold was defined and studied in \cite{ncManifold}. It was shown that non-commutative symmetric functions in two variables are in bijective correspondence with holomorphic functions on an nc manifold with matrix dimension of two, which suggests the potential of using nc manifolds to study geometric properties of nc functions.
		\newline
		In the context of non-commutative function theory, the \textit{n-dimensional non-commutative universe} $\mathbb{M}^n$ is the space of n-tuples of square matrices, that is 
		\begin{equation*}
			\mathbb{M}^n = \bigcup_{i=1}^{\infty}\mathbb{M}_d^n
		\end{equation*}
		where $\mathbb{M}_d$ denotes the $d\times d$ matrices over $\mathbb{C}$ and 
		\begin{equation*}
			\mathbb{M}_d^n = \{(x^1,\ldots,x^n)|\; x^r \in \mathbb{M}_d \text{ for all r} \}
		\end{equation*}
		By a free polynomial in n variables we mean a linear combination of words in n non-commuting variables, for example $p(X,Y) = I + XY - 2YX +XYX$ is a free polynomial.
		By a nc algebraic set we mean the zero locus in $\mathbb{M}_d^n$ of a collection of free polynomials. In the scalar case, which could be considered as $d=1$, an algebraic set is a manifold if some strictly lower dimensional singularities are removed, which is a consequence of implicit function theorem.
		\newline
		In the nc setting, there is also an implicit function theorem \cite{ncIFT}. When simplified to a single free polynomial in the $\mathbb{M}_n^2$ case, it says that the zero locus of $f(X,Y)$ can be locally written as the graph of a single variable function, that is $Z_f = \{(X,g(X))\}$ when 
		\begin{equation*}
			\lim_{t \rightarrow 0} \frac{f(X,Y+tH)-f(X,Y)}{t} \neq 0 \qquad \forall H \neq 0
		\end{equation*}
		This nc version of implicit function theorem suggests that the zero locus is generically an nc manifold except at some isolated singular point, as in the scalar case.
		\newline
		However, the purpose of this paper is to show the zero locus is not a nc manifold. Consider the algebraic set $V = \{(X,Y)|XY+YX=0\}$. We shall show that as a submanifold of $\mathbb{C}^{2d^2}$, it has dimension $d^2 + \lfloor \frac{d}{2}\rfloor$, therefore it cannot locally be $\{f(Z),g(Z)|Z\in \mathbb{M}_d\}$ for any functions $f$ and $g$. This follows from our main result
		\begin{lemma}\label{q ru}
			Let $q$ be a root of unity with $q^n = 1$. The non-commutative zero locus of $\{XY-qYX=0\}$ in $\mathbb{M}_d^2$ has dimension $d^2 + \lfloor\frac{d}{n}\rfloor$, and the commutative zero locus have dimension $d^2$.
		\end{lemma}
		When $q$ is not a root of unity, we show that
		\begin{lemma}\label{q nru}
			Let $q$ be an arbitrary complex number that is not a root of unity. Both the non-commutative zero locus of $\{XY-qYX=0\}$ and the commutative zero locus in $\mathbb{M}_d^2$ has dimension $d^2$.
		\end{lemma}
		Notice that the general second order homogeneous nc polynomial in two variables can be reduced to the situation of the previous two theorems. We have the following corollary
		\begin{theorem}\label{main}
			Consider the second order homogeneous nc polynomial $aX^2+bXY+cYX+dY^2$ in $\mathbb{M}_d^2$. Let $\Delta = (b+c)^2-4ad$
			\begin{enumerate}[(i)]
				\item If $\Delta \neq 0$ and $n$ is the smallest integer such that $(-\frac{\sqrt{\Delta}+b-c}{\sqrt{\Delta}-b+c})^n=1$, the non-commutative zero locus has dimension $d^2 + \lfloor \frac{d}{n}\rfloor$, and the commutative zero locus has dimension $d^2$;
				\item If $\Delta \neq 0$ and $-\frac{\sqrt{\Delta}+b-c}{\sqrt{\Delta}-b+c}$ is not a root of unity, the non-commutative zero locus has dimension $d^2$, and so does the commutative zero locus;
				\item If $\Delta = 0$, all solutions commute and the dimension is $d^2$
			\end{enumerate}
		\end{theorem}
	\section{NC Manifolds}
	In this section, we give the formal definition of nc-manifolds.
	The idea of nc-functions is that they preserve the basic properties of free polynomials. For a free polynomial $p(x^1,\ldots,x^n)$ in $n$ variables, the properties are 
	\begin{enumerate}[(1)]
		\item (graded) For $X = (X^1,\ldots,X^n)\in \mathbb{M}_d^n $, $P(X) \in \mathbb{M}_d$.
		\item (preserve direct sum) For $X  = (X^1,\ldots,X^n)\in \mathbb{M}_{d_1}^n, Y = (Y^1,\ldots,Y^n)\in \mathbb{M}_{d_2}^n$, let $X \oplus Y = (X^1\oplus Y^1,\ldots,X^n\oplus Y^n)$, then $p(X \oplus Y) = p(X) \oplus p(Y)$.
		\item (preserve joint similarity) For $X = (X^1,\ldots,X^n)\in \mathbb{M}_d^n $ and $S \in GL(\mathbb{C},d)$, let $SXS^{-1} = (SX^1S^{-1},\ldots,SX^nS^{-1})$, then $p(SXS^{-1}) = Sp(X)S^{-1}$.
	\end{enumerate}
	\begin{definition}
		(NC sets) If $D \in \mathbb{M}^d$ we say that $D$ is an nc-set if $D$ is closed with respect to direct sums, that is:\newline
		For all $d_1,d_2 \in \mathbb{N}$, and all $X_1 \in D\cap \mathbb{M}_{d_1}^n$, $X_2 \in D\cap \mathbb{M}_{d_2}^n$, 
		\begin{equation}
			X_1 \oplus X_2 \in D\cap \mathbb{M}_{d_1+d_2}^n
		\end{equation}
	\end{definition}
	\begin{definition}
		(NC functions) An nc-function in n variables is a function $f$, whose domain is an nc-set $D\subset \mathbb{M}^n$ and codomain is $\mathbb{M}^1$, which satisfies the property of graded, direct sum preserving and joint similarities preserving. 
		\newline
		A holomorphic nc-function is defined as the locally bounded nc-functions and a k-tuple of holomorphic  nc-functions is a nc-mapping.
	\end{definition}
	\begin{definition}(NC manifolds)
		If $X$ is a set, then we say that $\alpha$ is a n-dimensional coordinate patch or chart on $X$ if $\alpha$ is a bijection from a set $U_\alpha \subset X$ to a set $D_\alpha \subset \mathbb{M}^n$. If $\alpha$ and $\beta$ are a pair of n-dimensional coordinate patches on $X$ with $U_\alpha \cap U_\beta \neq \emptyset$, then the transition map $T_{\alpha\beta}$ is defined as 
		\begin{align*}
			&T_{\alpha\beta}:\alpha(U_\alpha \cap U_\beta) \rightarrow \beta(U_\alpha \cap U_\beta)\\
			&T_{\alpha\beta}(x) = \beta \circ \alpha^{-1}(x)
		\end{align*}
		A set $\mathcal{A}$ is called a nc-atlas for $X$ if $\mathcal{A}$ is a collection of n-dimensional coordinate patches on $X$,
		\begin{equation*}
			\cup_{\alpha\in \mathcal{A}} U_\alpha = X
		\end{equation*}
		and, for all $\alpha,\beta \in \mathcal{A}$,
		\begin{enumerate}[(1)]
			\item $\alpha(U_\alpha \cap U_\beta)$ is a union of nc-sets, and
			\item for every nc-subset $W$ of $\alpha(U_\alpha \cap U_\beta)$, the restriction of $T_{\alpha\beta}$ to $W$ is an nc-mapping.
		\end{enumerate} 
	\end{definition}
	From the definition we could conclude that an n-dimensional nc-manifold is locally isomorphic to $\mathbb{M}^n$.
	\section{Notations and lemmas}
		To show the main result, several lemmas are required. The first one is the uniqueness of Sylvester equations, whose proof can be found in \cite{matrixAnalysis}
		\begin{lemma}\label{sylvester}
			Let $A,B \in \mathbb{M}_d$. The Sylvester equation $AX+XB=C$ for any $C\in \mathbb{M}_d$ has a unique solution $X \in \mathbb{M}_d$ if and only if $A$ and $-B$ don't share any eigenvalue.
		\end{lemma}
		We introduce several notations about Jordan canonical forms.
		\begin{definition}
			$\mathcal{J}$ is defined to be the space of all $d\times d$ Jordan canonical forms. We hence define $\mathcal{J}_m$ to be the m-th skeleton of $\mathcal{J}$, such that for all Jordan canonical forms in $\mathcal{J}_m$, $m$ is the smallest integer such that $\sigma(J) \in \cup_{i=1}^{m} \{\lambda_i,  -\lambda_i\} \cup \{0\}$.
		\end{definition}
		For example, when $d=2$,
		\begin{equation*}
			\mathcal{J} = \{\begin{pmatrix}
				a &0\\0&b
			\end{pmatrix} \text{ or } \begin{pmatrix}
				c &1\\ 0&c
			\end{pmatrix}| a,b,c \in \mathbb{C}\}
		\end{equation*}
		\begin{equation*}
			\mathcal{J}_0 = \{\begin{pmatrix}
				0&0\\0&0
			\end{pmatrix},\begin{pmatrix}
				0&1\\0&0
			\end{pmatrix}\}
		\end{equation*}
		\begin{align*}
			\mathcal{J}_1 = \{\begin{pmatrix}	a&0\\0&a\end{pmatrix}|a\in \mathbb{C}\setminus{0}\} \cup
			\{\begin{pmatrix}	a&0\\0&-a\end{pmatrix}|a\in \mathbb{C}\setminus{0}\} \cup \{\begin{pmatrix}	a&0\\0&0\end{pmatrix}|a\in \mathbb{C}\setminus{0}\}\\ \cup \{\begin{pmatrix}	0&0\\0&a\end{pmatrix}|a\in \mathbb{C}\setminus{0}\} \cup \{\begin{pmatrix}	a&1\\0&a\end{pmatrix}|a\in \mathbb{C}\setminus{0}\}
		\end{align*}
		and
		\begin{equation*}
			\mathcal{J}_2 = \mathcal{J}\setminus(\mathcal{J}_0\cup\mathcal{J}_1)
		\end{equation*}
		Notice that $\mathcal{J}_m$ consists of components and the intersection of the components are contained in lower skeletons, so the components are disjoint. 
		%To characterize these components, let $a^{(m)}$ be a map from a m-tuple to d-tuple, and let $b^{(m)}$ be a (d-1)-tuple of 0's and 1's, then $J(a^{(m)},b^{(m)})$ is the Jordan canonical form with diagonal  and superdiagonal being $b^{(m)}$.
		\newline
		We introduce this strange looking notation since the solution of $XY + qYX=0$ has different dimensions when the Jordan canonical form of $Y$ lies on $\mathcal{J}_m$. And we are going to show a subset of the solution of $XY+YX=0$ attached to one of the components in $\mathcal{J}_{\lfloor \frac{d}{2} \rfloor}$ has dimension $d^2+\lfloor \frac{d}{2} \rfloor$ and is the one with highest dimension.
		\newline
		Also
		 we need two inequalities regarding finite sequences.
		\begin{lemma}\label{inequality}
			For two finite sequence of numbers $m_1,\ldots,m_k$ and $n_1,\ldots,n_l$,
			\begin{equation}\label{ineq}
				\sum_{i,j=1}^{k} \min\{m_i,m_j\} + \sum_{i,j=1}^{l}\min\{n_i,n_j\} \geq 2\sum_{i=1}^{k}\sum_{j=1}^{l}\min\{m_i,n_j\}
			\end{equation}
			The equality holds only when $k=l$ and $m_i = n_i$ for all $i$.
		\end{lemma}
        
		\begin{proof}
			With out changing both sides, add $0$ to the shorter sequence to make $k=l$, and rearrange both sequences in increasing order. Then we prove the result by induction on the length of the sequences.
			\newline
			When both sequences have length equal to 1, then $m_1 + n_1 \geq 2\min\{m_1,n_1\}$.
			\newline
			Suppose the inequality holds for two sequences with length $k$, then consider two sequences with length $k+1$.
			\begin{align*}
				\sum_{i,j=1}^{k+1} \min\{m_i,m_j\} +& \sum_{i,j=1}^{k+1}\min\{n_i,n_j\}\\
				=&(\sum_{i,j=1}^{k} \min\{m_i,m_j\} + 2\sum_{i=1}^{k}m_i + m_{k+1}) \\+& (\sum_{i,j=1}^{l}\min\{n_i,n_j\} + 2\sum_{i=1}^{k}n_i + n_{k+1})\\
			\end{align*}
			While
			\begin{align*}
				\sum_{i,j=1}^{k+1}\min\{m_i,n_j\}
				=&\sum_{i,j=1}^{k}\min\{m_i,n_j\} + \sum_{i=1}^{k}\min\{m_i,n_{k+1}\} \\+& \sum_{i=1}^{k}\min\{m_{k+1},n_i\} + \min\{m_{k+1},n_{k+1}\}
			\end{align*}
			Since
			\begin{align*}
				\sum_{i=1}^{k}\min\{m_i,n_{k+1}\} &< \sum_{i=1}^{k}m_i \\
				\sum_{i=1}^{k}\min\{m_{k+1},n_i\} &< \sum_{i=1}^{k}n_i \\
				m_{k+1} + n_{k+1} &< 2\min\{m_{k+1},n_{k+1}\}
			\end{align*}
			the induction is completed.
		\end{proof}
		Applying Lemma (\ref{inequality}) for $l$ times on pairs of sequences 
		\begin{equation*}
			\{\{m_i^1\}_i,\{m_i^2\}_i\}, \ldots, \{\{m_i^{l-1}\}_i,\{m_i^l\}_i\}\text{, and } \{\{m_i^l\}_i,\{m_i^1\}_i\}
		\end{equation*}
		then adding them together, we get the following general lemma.
		\begin{lemma}\label{geninequality}
			For finitely many finite sequence of numbers $m_{1}^1,\ldots m_{k_1}^1$, $\ldots$, $m_{1}^l,\ldots m_{k_l}^l$,
			\begin{align}
				\begin{split}
					&\sum_{i,j = 1}^{k_1} \min\{m_i^1,m_j^1\} + \cdots + \sum_{i,j = 1}^{k_l} \min\{m_i^l,m_j^l\}  \\
					\geq&\sum_{i=1}^{k_1}\sum_{j=1}^{k_2} \min\{m_i^1,m_j^2\} + \cdots +\sum_{i=1}^{k_{l-1}}\sum_{j=1}^{k_l} \min\{m_i^{l-1},m_j^l\} + \sum_{i=1}^{k_l}\sum_{j=1}^{k_1} \min\{m_i^l,m_j^1\}
				\end{split}
			\end{align}
		\end{lemma}
	\section{Proofs of the main theorems}

		\begin{proof}[Proof of Theorem 1]    
			$ $\newline
			\textbf{I. A special case.}
			Assuming first that we are in the special case when $q = -1$. Since we are trying to find the nontrivial solutions other than $(0,0)$, we assume the criteria in Lemma \ref{sylvester} fails, that is, $Y$ has at least one pair of eigenvalues being in the form $\{\lambda, -\lambda\}$ or one eigenvalue being $0$. Write $Y$ in its Jordan decomposition $Y=P^{-1}JP$, where $J$ is the Jordan canonical form of $Y$. Then
			\begin{align}\label{diag}
				0=XY+YX&=XP^{-1}JP+P^{-1}JPX\nonumber \\&=P^{-1}(PXP^{-1}J +JPXP^{-1})P
			\end{align}
			Let $\eta(J)$ be the smallest integer such that $\sigma(Y) = \sigma(J) \in \cup_{i=1}^{\eta} \{\lambda_i,  -\lambda_i\} \cup \{0\}$. Then $J$ is in form of 
			\begin{equation*}
				J = \bigoplus_{i=1}^{\eta} \bigoplus_{k=1}^{p_i} J_{d_{i,k}}^{(\lambda_i)} \oplus \bigoplus_{i=1}^{\eta}\bigoplus_{l=1}^{q_i} J_{d_{i,l}'}^{(-\lambda_i)} \oplus \bigoplus_{j=1}^{r} J_{d_{j}}^{(0)},
			\end{equation*}
			where $J_l^\lambda$ is a Jordan block with size $l \times l$ and $\lambda$ on its diagonal. That is
			\begin{equation*}
				J 
				=\begin{pmatrix}
					\begin{matrix}
						\lambda_1 & \alpha_1 & 0\\
						0 & \lambda_1 & \ddots\\
						0 & 0 & \lambda_1
					\end{matrix}
					&\mathbf{0}&\mathbf{0}&\mathbf{0} \\
					\mathbf{0}& 
					\begin{matrix}
						-\lambda_1 & \beta_1 & 0\\
						0 & -\lambda_1 & \ddots\\
						0 & 0 & -\lambda_1
					\end{matrix}
					&\mathbf{0}&\mathbf{0}\\
					\mathbf{0}&\mathbf{0}& \ddots &\mathbf{0}\\
					\mathbf{0}&\mathbf{0}&\mathbf{0}&
					\begin{matrix}
						0 & \gamma & 0\\
						0 & 0 & \ddots\\
						0 & 0 & 0
					\end{matrix}
				\end{pmatrix}_{n \times n}
			\end{equation*}
			%where $\mathbf{a}(\lambda)$ is the diagonal of $J$ where $\mathbf{a}$ maps an m-tuple $\lambda$ to a d-tuple, and $\mathbf{b}$ is the superdiagonal of $J$ that is an (n-1)-tuple consist of 0's and 1's. 
			Let $\mathcal{J}_m$ be the set of all $J$ with $\eta(J) =m$ for $m$ being a fixed integer, then $\mathcal{J}_m$ forms a m-dimensional complex manifold, hence all possible $d \times d$ Jordan canonical form lies in $\mathcal{J} = \cup_{i=0}^d \mathcal{J}_m$. Then all the matrices $Y$ with Jordan canonical form on $\mathcal{J}$ is $\mathcal{Y} \cong \mathcal{J} \times SL_d(\mathbb{C})/\sim$ where the relation $\sim$ is that $(J_1,P_1) \sim (J_2,P_2)$ when $P_1JP_1^{-1} = P_2P(J)P_2^{-1}$, where $P(J)$ is a permutation of the Jordan blocks in $J$. Considering (\ref{diag}), we claim that there is a bijective relationship between
			\begin{equation}
				\{(X,Y)|XY+YX=0\} \text{ and } \{(Q,J,P)| QJ+JQ = 0, J\in \mathcal{J}, P \in SL_d(\mathbb{C})\}/\sim
			\end{equation}
			Here the relation is $(Q_1, J, P_1) \sim (P_2P_1^{-1}Q_1P_1P_2^{-1
			}, J, P_2)$. In terms of dimension counting, the dimension for $P$ is $n^2$, therefore it remains to consider the dimension of $J$ and $Q$.

			Separate $J$ by sizes of Jordan block on diagonal, such that the $i,j$-th block of $J$ has size $a \times b$ given the i-th Jordan block is $a \times a$ and j-th Jordan block is $b \times b$. For $Q$ that satisfies $QJ+JQ=0$, separate $Q$ in the same way as $J$ does, then the i,j-th block of $Q$, $Q_{i,j}$, is the solution of $XJ_m^{\lambda} + J_n^{\mu}X=0$, where $J_m^{\lambda}$ is the j-th diagonal block of $J$ and $J_n^{\mu}$ is the i-th diagonal block of $J$. Again by Lemma \ref{sylvester}, functions in form of $XJ_m^{\lambda} + J_n^{\mu}X=0$ has a unique solution $X=0$ when $\lambda+\mu \neq 0$. 
			
			For the equivalence relation, we need to quotient out the commutant for each fixed $J$. Let $C \in \mathbb{M}_d$ such that $CJ-JC=0$. Again cut $C$ into blocks by the sizes of Jordan blocks of $J$, then $C_{i,j}J_m^{\lambda} - J_n^{\mu}C_{i,j}=0$ has non zero solution only when $\lambda = \mu$
			
			Therefore if we cut $J$ ,$Q$ and $C$ into $(m+1) \times (m+1)$ blocks of the same sizes with the i-th diagonal block of $J$ only have $\lambda_i, -\lambda_i$ on diagonal, then $Q$ and $C$ is nonzero only on the diagonal blocks. So when we are considering the dimension of $\{Q| QJ+JQ = 0\}$ and $\{C|CJ-JC=0\}$ for some fixed $ J\in \mathcal{J}$, it suffice to consider $J,Q,C$ block by block along diagonal with $J$'s of the form $\bigoplus_{k=1}^p J_{d_k}^{(\lambda)} \oplus \bigoplus_{l=1}^{q}J_{d_{l}'}^{(-\lambda)}$ or $\bigoplus_{i=1}^{m}J_{d_i}^{(0)}$ and then add up the dimensions. 
			\newline
			\textbf{Case 1} Consider $J$'s that lie in the submanifold of $\mathcal{J}$ given by $\mathcal{J}^{(\lambda)} = \{\bigoplus_{k=1}^p J_{d_k}^{(\lambda)} \oplus \bigoplus_{l=1}^{q}J_{d_{l}'}^{(-\lambda)} | \lambda \neq 0\} \subset \mathbb{M}_d$. Then $\mathcal{J}^{(\lambda)}$ is a 1 complex dimensional manifold that only depends on $\lambda \in \mathbb{C}$, that is
			\begin{equation} \label{dimJ}
				\dim(\mathcal{J}^{(\lambda)}) = 1
			\end{equation}
			For the commutant $\{C|CJ=JC\}$ of fixed $J \in \mathcal{J}^{(\lambda)}$, cut $C$ into blocks by the sizes of Jordan blocks of $J$, then since $C_{i,j}J_m^{\lambda} - J_n^{\mu}C_{i,j}=0$ has non zero solution only when $\lambda = \mu$, $C$ can be nonzero only in i,j-th blocks for $1\leq i,j \leq p$ or $p+1\leq i,j \leq d$. Since the solution of $X$ such that $XJ_m^\lambda-j_n^\lambda X=0$ is in form of 
			\begin{equation}
				\begin{matrix}
					a &b &c &\cdots \\
					0 &a &b &\cdots \\
					0 &0 &a &\cdots \\
					\vdots &\vdots &\vdots &\ddots
				\end{matrix}
			\end{equation}
			the solution of $X$ has dimension $\min \{m,n\}$. Therefore for fixed $J$
			\begin{equation}\label{dimC}
				\dim(\{C|CJ=JC\}) = \sum_{k_1=1}^{p}\sum_{k_2=1}^{p} \min\{d_{k_1},d_{k_2}\} + \sum_{l_1=1}^{q}\sum_{l_2=1}^{q} \min\{d_{l_1}',d_{l_2}'\}
			\end{equation}

			It remains to find $\dim(\{Q|QJ+JQ=0\})$. Again cut $Q$ into blocks by the sizes of Jordan blocks of $J$, then since $C_{i,j}J_m^{\lambda} + J_n^{\mu}C_{i,j}=0$ has non zero solution only when $\lambda = -\mu$, $Q$ can be nonzero only in i,j-th blocks for $1\leq i \leq p, p+1\leq j \leq d$ or $p+1\leq i \leq d, 1\leq j \leq p$. Since the solution of $X$ such that $XJ_m^\lambda+j_n^{-\lambda} X=0$ is in form of 
			\begin{equation}
				\begin{matrix}
					a &b &c &\cdots \\
					0 &-a &-b &\cdots \\
					0 &0 &a &\cdots \\
					\vdots &\vdots &\vdots &\ddots
				\end{matrix}
			\end{equation}
			the solution of $X$ has dimension $\min \{m,n\}$. Therefore for fixed $J$
					
			\begin{equation}\label{dimQ}
				\dim(\{Q|QJ+JQ=0, J \in \mathcal{J}_\lambda\}) = 2\sum_{k=1}^{p}\sum_{l=1}^{q}\min\{d_k,d_l'\}
			\end{equation}
			Combining (\ref{dimJ}),(\ref{dimC}),(\ref{dimQ}), there is 
			\begin{align}
				\begin{split}
					\dim(\{Q|QJ+JQ=&0, J \in \mathcal{J}_\lambda\}) +
					\dim(\mathcal{J}_\lambda) - \dim(\{C|CJ=JC, J \in \mathcal{J}_\lambda\})\\
					=&2\sum_{k=1}^{p}\sum_{l=1}^{q}\min\{d_k,d_l'\} +1 \\
					-& \sum_{k_1=1}^{p}\sum_{k_2=1}^{p} \min\{d_{k_1},d_{k_2}\} - \sum_{l_1=1}^{q}\sum_{l_2=1}^{q} \min\{d_{l_1'},d_{l_2'}\} \\
					\leq& 1
				\end{split}, 
			\end{align}
			where the last inequality is given by Lemma \ref{inequality}.
			\newline
			\textbf{Case 2} suppose $ \mathcal{J}_0= \{\bigoplus_{i=1}^{m}J_{d_i}^{(0)}\} \subset \mathbb{M}_d$. Then
			\begin{align}
				&\dim(\mathcal{J}_0) = 0\\
				&\dim(\{C|CJ=JC, J \in \mathcal{J}_0\}) = \sum_{i,j=1}^{m} \min\{d_i,d_j\}\\
				&\dim(\{Q|QJ+JQ=0, J \in \mathcal{J}_0\}) = \sum_{i,j=1}^{m} \min\{d_i,d_j\}
			\end{align}
			Therefore $\dim(\{Q|QJ+JQ=0, J \in \mathcal{J}_0\}) + \dim(\mathcal{J}_0) - \dim(\{C|CJ=JC, J \in \mathcal{J}_0\}) =0$.
			\newline
			From the two cases above, each pair of $\{\lambda, -\lambda\}\in \sigma(Y)$ gives at most 1 additional dimension to the overall dimension of zero locus $\{(X,Y)|XY+YX=0\}$. The component of 
			\begin{equation*}
				\{(X,Y)|XY+YX=0 | \text{ Jordan canonical form of Y lies in }\mathcal{J}_m\}
			\end{equation*}
			with largest dimension is $\{Q|QJ+JQ=0, J \in \mathcal{J}_*\} \times \mathcal{J}_* \times GL_d(\mathbb{C}) / \{C|CJ=JC, J \in \mathcal{J}_*\} $ where $\mathcal{J}_*$ is $\mathcal{J}_* = \{\bigoplus_{i=1}^{m}\bigoplus_{k=1}^{p_i}(J_{d_{i,k}}^{(\lambda_i)} \oplus J_{d_{i,k}}^{(-\lambda_i)})\oplus J_{d_0}^{(0)}\in \mathbb{M}_d\}$, and the corresponding $D = n^2 + m \leq n^2 + \lfloor \frac{n}{2} \rfloor$, where equality is realized when all Jordan blocks are $1 \times 1$, hence $Y$ is diagonalizable. In this case the component with largest dimension is in form of
			\begin{align*}
				PYP^{-1} = J = diag(\lambda_1,-\lambda_1,\lambda_2,-\lambda_1,\ldots,\lambda_{d/2},-\lambda_{d/2})\\
				PXP^{-1} = (\begin{pmatrix} 0& z_1\\w_1&0\end{pmatrix})\oplus(\begin{pmatrix} 0& z_2\\w_2&0\end{pmatrix})\oplus \cdots \oplus (\begin{pmatrix} 0& z_{d/2}\\w_{d/2}&0\end{pmatrix})
			\end{align*}
			\newline
			\textbf{II The general case} Now consider the general case when $q^n=1$. By a similar argument to \textbf{I}, in order for $XY+qYX=0$ to have a non-trivial solution, the eigenvalues of $Y$ must either have a pairs of the form $\{\lambda,q\lambda\}$ or a 0 eigenvalue. Now let $m$ be the smallest integer such that $\sigma(J) \in \cup_{i=1}^m\{\lambda_i,q\lambda_i,\ldots,q^{n-1}\lambda_i\}$, then
			\begin{equation*}
				J = \bigoplus_{i=1}^{m}\Bigg( \bigoplus_{k=1}^{p_i^{(0)}} J_{d_{i,k}^{(0)}}^{(\lambda_i)} \oplus \bigoplus_{l=1}^{p_i^{(1)}} J_{d_{i,l}^{(1)}}^{(q\lambda_i)} \oplus \cdots \oplus
				\bigoplus_{l=1}^{p_i^{(n-1)}} J_{d_{i,l}^{(n-1)}}^{(q^{(n-1)}\lambda_i)}\Bigg)
				\oplus \bigoplus_{j=1}^{r} J_{d_{j}}^{(0)}.
			\end{equation*}Note that some of the $d_{\alpha,\beta}^{(\gamma)}$ can be 0. Parallel with the same reasoning in the special case, we could consider $J$ block by block, since most of the entries of $Q$ and commutant of $J$ is zero off diagonal.
			\newline
			\textbf{Case 1'} Consider the submanifold of $\mathcal{J}_\lambda$ that $\mathcal{J}_\lambda = \{\bigoplus_{k=1}^{p^{(0)}}J_{d_k^{(0)}}^{(\lambda)} \oplus \cdots \oplus \bigoplus_{k=1}^{p^{(n-1)}}J_{d_k^{(n-1)}}^{(q^{n-1}\lambda)}\} \subset \mathbb{M}_d$. Then for fixed $J \in \mathcal{J}_\lambda$,
			\begin{align}
				&\dim(\mathcal{J}_\lambda) = 1 \\
				&\dim(\{C|CJ=JC\}) \nonumber \\ 
				&\qquad= \sum_{k,l=1}^{p^{(0)}} \min\{d_{k}^{(0)},d_{l}^{(0)}\} + \cdots +
				\sum_{k,l=1}^{p^{(n-1)}} \min\{d_{k}^{(n-1)},d_{l}^{(n-1)}\} \\
				&\dim(\{Q|QJ-qJQ=0\}) \nonumber\\
				&\qquad = \sum_{k=1}^{p^{(0)}}\sum_{k=1}^{p^{(1)}}\min\{d_{k}^{(0)},d_{l}^{(1)}\}+\cdots
				+\sum_{k=1}^{p^{(n-2)}}\sum_{k=1}^{p^{(n-1)}}\min\{d_{k}^{(n-2)},d_{l}^{(n-1)}\} \nonumber\\
				&\qquad +\sum_{k=1}^{p^{(n-1)}}\sum_{k=1}^{p^{(0)}}\min\{d_{k}^{(n-1)},d_{l}^{(0)}\}
			\end{align}
			Therefore
			\begin{align}\label{general}
				\begin{split}
					&\dim(\{Q|QJ-qJQ=0, J\in \mathcal{J}_\lambda\}) + \dim(\mathcal{J}_\lambda) - \dim(\{C|CJ=JC, J\in \mathcal{J}_\lambda\})\\
					&\qquad= \sum_{k=1}^{p^{(0)}}\sum_{k=1}^{p^{(1)}}\min\{d_{k}^{(0)},d_{l}^{(1)}\}+\cdots
					+\sum_{k=1}^{p^{(n-2)}}\sum_{k=1}^{p^{(n-1)}}\min\{d_{k}^{(n-2)},d_{l}^{(n-1)}\}\\
					&\qquad+\sum_{k=1}^{p^{(n-1)}}\sum_{k=1}^{p^{(0)}}\min\{d_{k}^{(n-1)},d_{l}^{(0)}\} + 1\\
					&\qquad- (\sum_{k,l=1}^{p^{(0)}} \min\{d_{k}^{(0)},d_{l}^{(0)}\} + \cdots +
					\sum_{k,l=1}^{p^{(n-1)}} \min\{d_{k}^{(n-1)},d_{k}^{(n-1)}\})\\
					&\qquad\leq 1,
				\end{split}
			\end{align}
			by Lemma (\ref{geninequality}), equality holds when $p^{(0)}=\cdots=p^{(n-1)}>0$ and all $d_\alpha^{(\beta)}$'s are the same for all $\beta$'s and fixed $\alpha$.
			\newline
			\textbf{Case 2'} Let $\mathcal{J}_0 = \{\bigoplus_{i=1}^{m}J_{d_i}^{(0)}\}\subset \mathbb{M}_d$. Then $\dim(\{Q|QJ-qJQ=0, J \in \mathcal{J}_0\}) + \dim(\mathcal{J}_0) - \dim(\{C|CJ=JC, J\in \mathcal{J}_0\}) =0$.
			\newline
			Thus, the component of $\{(X,Y)|XY+YX=0\}$ with largest dimension is $\{Q|QJ+JQ=0, J \in \mathcal{J}_*\} \times \mathcal{J}_* \times GL_d(\mathbb{C}) / \{C|CJ=JC, J \in \mathcal{J}_*\} $ where 
			\begin{equation}
				\mathcal{J}_* = \{\bigoplus_{i=1}^{m}\bigoplus_{k=1}^{p}(J_{d_i,k}^{(\lambda_i)}\oplus\cdots\oplus J_{d_i,k}^{(q^{(n-1)}\lambda_i)})\oplus J_{d_0}^{(0)} \in \mathbb{M}_d\}, 
			\end{equation}and the corresponding $D = d^2 + m \leq d^2 + \lfloor \frac{d}{n}\rfloor$ with equality holds when, in addition, all Jordan blocks are $1 \times 1$. In this case the component with largest dimension is in form of
			\begin{equation*}
			PYP^{-1} = J = 			diag(\lambda_1,q\lambda_1, \ldots q^{n-1}\lambda_1,\ldots,\lambda_{d/n},q\lambda_{d/n},\ldots,q^{n-1}\lambda_{d/n})
			\end{equation*}
			\begin{equation*}
			PXP^{-1} = \begin{pmatrix} 0& z_1^1&&\\&0&z_1^2&\\&& \ddots&\\z_1^{n}&&&0 \end{pmatrix}\oplus\begin{pmatrix} 0& z_2^1&&\\&0&z_2^2&\\&& \ddots&\\z_2^{n}&&&0\end{pmatrix}\oplus \cdots \oplus \begin{pmatrix} 0& z_{d/n}^1&&\\&0&z_{d/n}^2&\\&& \ddots&\\z_{d/n}^{n}&&&0\end{pmatrix}
			\end{equation*}
		
		\end{proof}
		Following the proof above, we can also solve the case when $q$ is not a root of unity.
		\begin{proof}[Proof of the Theorem 2]
			When $q$ is not a root of unity, there isn't term $\sum_{k=1}^{p^{(n-1)}}\sum_{k=1}^{p^{(0)}}\min\{d_{k}^{(n-1)},d_{l}^{(0)}\}$ in (\ref{general}), which makes it strictly less than 1, therefore the dimension of $Z(\{(X,Y)|XY-qYX=0\})$ is at most $d^2$.
		\end{proof}
	
		Now consider the commutative case.
		\begin{theorem}[commutative case]
			The commutative zero locus of $XY-qYX$ in $\mathbb{M}_d^2$ has dimension $d^2$.
		\end{theorem}
		\begin{proof}
			Since $X,Y$ are commuting, $\{(X,Y)|XY-qYX = 0, XY-YX=0\} = \{(X,Y)|XY=YX=0\}$.
			Again consider the Jordan canonical form of $X$. Suppose $X = P^{-1}JP$ then $J(PYP^{-1}) = (PYP^{-1})J=0$. Since Jordan block with non-zero diagonal has full rank, the corresponding row and column of $Q := PYP^{-1}$ have to be zero at every entry. For Jordan block $J_d^{(0)}$, $\{Q|QJ_d^{(0)} = J_d^{(0)}Q=0\}$ is the $d \times d$ matrix with only $1,d$ entry can be non-zero.
			Therefore for $X$ with rank $r$, $\dim(\{Y| XY=YX=0\})$ is $(d-r)^2$. In addition that the dimension of matrices with rank $r$ is $rd + r(d-r) = 2rd-r^2$, the total dimension of the commutative zero locus is $2rd-r^2 + (d-r)^2 = d^2$.
		\end{proof}
	
		By changing of variables, we could extend our results to homogeneous quadratic equations for the main theorem.
\iffalse
		\begin{cor}
			Consider the second order homogeneous nc polynomial $aX^2+bXY+cYX+dY^2$ in $\mathbb{M}_d^2$. Let $\Delta = (b+c)^2-4ad$
			\begin{enumerate}[(i)]
				\item If $\Delta \neq 0$ and $n$ is the smallest integer such that $(-\frac{\sqrt{\Delta}+b-c}{\sqrt{\Delta}-b+c})^n=1$, the non-commutative zero locus has dimension $d^2 + \lfloor \frac{d}{n}\rfloor$, and the commutative zero locus has dimension $d^2$;
				\item If $\Delta \neq 0$ and $-\frac{\sqrt{\Delta}+b-c}{\sqrt{\Delta}-b+c}$ is not a root of unity, the non-commutative zero locus has dimension $d^2$, and so does the commutative zero locus;
				\item If $\Delta = 0$, all solutions commute and the dimension is $d^2$
			\end{enumerate}
		\end{cor}
\fi
		\begin{proof} [Proof of \ref{main}]
			Substituting $X,Y$ by $X=m_1 U + m_2 V$, $Y = m_1U + m_2V$, with coefficients $m_1,m_2,n_1,n_2$ to be determined, $X^2+bXY+cYX+dY^2$ becomes
			\begin{align*}
				&(am_1^2+(b+c)m_1n_1+dn_1^2)U^2 \\+& (am_1m_2+bm_1n_2+cm_2n_1+dn_1n_2)UV \\+&
				(am_1m_2+bm_2n_1+cm_1n_2+dn_1n_2)VU \\+& am_2^2+(b+c)m_2n_2+dn_2^2)V^2 =0
			\end{align*}
			In order to get rid of $U^2$ and $V^2$ terms while keep $X,Y$ non-commuting, $\frac{m_1}{n_1}$ and $\frac{m_2}{n_2}$ is chose to be the different roots of $az^2+(b+c)z+d=0$. Then $\{(X,Y)|aX^2+bXY+cYX+dY^2=0\} \simeq \{(U,V)|pUV+qVU=0\}$, by the change of variable $(U,V)^T = (\begin{matrix}
				m_1 &m_2\\
				n_1 &n_2
			\end{matrix})^{-1} (X,Y)^T$, where $p = am_1m_2+bm_1n_2+cm_2n_1+dn_1n_2$ and $q = am_1m_2+bm_2n_1+cm_1n_2+dn_1n_2$ are both non-zero since $aX^2+bXY+cYX+dY^2$ is irreducible. By Theorem (\ref{q ru}) and Theorem (\ref{q nru})corollary is proved.
		\end{proof}
		\begin{remark}
			Notice that if we rewrite $XY - qYX$ into $XY - f(Y)X$ and $f(Y) = qY$, then the $Y$ component of solution with largest dimension can be written as 
			\begin{equation*}
				PYP^{-1} = diag\big(\lambda_1, f(\lambda_1), f\circ f(\lambda_1),\ldots,f^{(n-1)}(\lambda_1),\ldots,\lambda_{d/n},\ldots,f^{(n-1)}(\lambda_{d/n})\big)
			\end{equation*}
			with $f^{(n)}$ denotes $f$ composes with itself for n times. And $X$ could have non-zero entry at n,1-entry for each block on diagonal because $f^{(n)} = id$. Since the Mobius transform $g(z) = \frac{az+b}{cz+d}$ satisfy $g^{(n)} = id$ when $\begin{pmatrix}  a&b\\c&d  \end{pmatrix}^n = I$, one could notice that 
			\begin{equation*}
				PYP^{-1} = diag\big(\lambda_1, f(\lambda_1), f\circ f(\lambda_1),\ldots,f^{(n-1)}(\lambda_1),\ldots,\lambda_{d/n},\ldots,f^{(n-1)}(\lambda_{d/n})\big)
			\end{equation*}
			and
			\begin{equation*}
				PXP^{-1} = \begin{pmatrix} 0& z_1^1&&\\&0&z_1^2&\\&& \ddots&\\z_1^{n}&&&0 \end{pmatrix}\oplus\begin{pmatrix} 0& z_2^1&&\\&0&z_2^2&\\&& \ddots&\\z_2^{n}&&&0\end{pmatrix}\oplus \cdots \oplus \begin{pmatrix} 0& z_{d/n}^1&&\\&0&z_{d/n}^2&\\&& \ddots&\\z_{d/n}^{n}&&&0\end{pmatrix}
			\end{equation*}
			forms a solution to $(cY+d)XY-(aY+b)X = 0$, which has dimension $d^2 + \lfloor \frac{d}{n} \rfloor$. Thus the zero locus of $(cY+d)XY-(aY+b)X$ has at least dimension $d^2 + \lfloor \frac{d}{n} \rfloor$.
		\end{remark}

\printbibliography
\end{document}